\documentclass[10pt]{article}
\usepackage[utf8]{inputenc}
\usepackage[ruled,linesnumbered]{algorithm2e}

\usepackage{amssymb}
\usepackage{amsmath}
\usepackage{amsthm}
\usepackage{url}
\usepackage{mathrsfs}
\usepackage{pdfpages}
\usepackage{epigraph}
\usepackage{tikz-cd}

\newtheorem{lem}{Lemma}
\newtheorem{deff}{Definition}
\newtheorem{mytheor}{Theorem}

\newtheorem{theor}{Theorem}
\newtheorem{cor}{Corollary}

\newcommand{\N}{\mathbb{N}}

\newcommand{\sA}{\mathscr{A}}

\DeclareMathOperator{\stab}{\mathbf{stab}}

\DeclareMathOperator{\pr}{pr}

\begin{document}

\author{Andrei Alpeev  \footnote{DMA, École normale supérieure, Université PSL, CNRS, 75005 Paris, France,\\alpeevandrey@gmail.com}}
\title{Non-C*-simple groups admit non-free actions on their Poisson boundaries}

\maketitle
\begin{abstract}
It is a classical result of Kaimanovich and Vershik and independently of Rosenblatt that a non-amenable group admits a non-degenerate symmetric measure such that the Poisson boundary is trivial. Most if not all examples to date of non-free actions of countable groups on their Poisson boundaries had the stabilizers sitting inside the amenable radical. We show that every countable non-C*-simple group admits a symmetric measure of full support with non-trivial stabilizers. For a class of non-C*-simple groups with trivial amenable radical, which is non-empty as was shown by le Boudec, this gives a wealth of examples with non-normal stabilizers.
\end{abstract}
I write this note to demonstrate a serie of mildly peculiar examples in the realm of the Poisson boundaries. One motivation for the current work is the question, whether it is possible for the action of the group on its Poisson boundary (corresponding to some non-degenerate measure on the group) to have different stabilizers for different points. Another, is to make a connection between properties of the Posiison bounary and the Furstenberg boundary. Two properties of groups related to their C*-algebras turned out to be amenable to analysis by the way of considering their Furstenberg boundary: that of C*-simplicity and of the unique trace property. Initially, this connection was drawn for C*-simplicity by Kalantar and Kennedy in \cite{KaKe17}. 
It is said that a group has {\em unique trace property} if there is a unique trace (the canonical trace) on its reduced C*-algebra. Breuillard, Kallantar, Kennedy and Ozawa in \cite{BKKO17} proved that 
\begin{theor}
A group has unique trace property iff it has no non-trivial amenable subgroups. 
\end{theor}

A group is called {\em C*-simple} if its reduced C*-algebra does not have non-trivial factors.
A group $G$ is not {\em C*-simple} iff it has an amenable subgroup $H$ and a finite subset $S$ not containing the group identity such that $g^{-1}Hg \cap S \neq \varnothing$ for every $g \in G$, as was shown by Kennedy \cite{Ke20}. 
%\begin{mytheor}\label{thm: main}
%On every non-C*-simple group there is a symmetric measure of full support such that the action of $G$ on the Poisson-Furstenberg boundary is not essentially free. Moreover, the stabilizer of almost every point has non-trivial intersection with $K$ from the definition of C*-simplicity
%\end{mytheor} 

\begin{deff}
A finite subset $S$ of a group $G$ is called {\em amenably-visible} if for some amenable subgroup $H$ of $G$ we have that $S \cap H^\gamma \neq \varnothing$ for all $\gamma \in G$ (equivalently $S^\gamma \cap H \neq \varnothing$). 
\end{deff}
There are some easy examples. First, any subset that contains the group identity is trivially amenably-visible. Any subset that has a non-empty intersection with a normal amenable subgroup is amenably visible.

The main theorem of \cite{Ke20} could be trivially reformulated in the following way:

\begin{theor}\label{thm: inner criterion}
A group is not C*-simple iff it has a finite amenably-visible subset not containing the group identity.
\end{theor}

%We actually prove that following:
The main result of this paper is the following:

\begin{mytheor}\label{thm: main}
Let $G$ be a countable non C*-simple group. There is a symmetric measure $\nu$ of full support $G$ such that  for every amenably-visible subset $S$ of $G$ and for almost every point $\xi$ of the Poisson-Furstenberg boundary, the stabilizer $\stab_G(\xi)$ has non-empty intersection with $S$.
\end{mytheor}

Note that this theorem has the Kaimanovich-Vershik and Rosenblatt result as a corollary. Indeed, for for an amenable group $G$ all one-element subsets are amenably visible, so the measure constructed should have the whole group as a stabilizer of almost every point in the boundary.

Le Boudec proved the following in \cite{Bo17}:

\begin{theor}
There are non-C*-simple groups that do not have non-trivial amenable subgroups.
\end{theor} 

For our purposes we are interested in countable groups, so we need the following lemma:

\begin{lem}
If $H$ is a countable subgroup of a group $\Gamma$ and $\gamma$ is a non-C*-simple group without non-trivial amenable subgroup, then there is a subgroup $G$of $\Gamma$ that contains $H$ and such that $G$ is countable.
\end{lem}

\begin{proof}
By theorem \ref{thm: inner criterion}, there is a finite amenably-visible subset $S$ of $\Gamma$ that is disjoint from the identity. Note that by the same theorem, any subgroup of $\Gamma$ that contains $S$ will be non-C*simple. Let $G_1$ be the subgroup of $\Gamma$ generated by $H$ and $S$. Note that since $\Gamma$ has no non-trivial amenable normal subgroups, for any non-trivial element $a \in \Gamma$, the normal subgroup generated by it is non-amenable. This implies that there are finitiely-many elements $t_1, \ldots, t_k$ of $\Gamma$ such that the subgroup generated by $a^{t_1}, \ldots, a^{t_k}$ is non-amenable, we call such finite sets  a witness-set for $a$. Now, we inductively define $G_{i+1}$ by adding a finite witness-set for each non-trivial elements of $G_i$, and considering the generated subgroup of $\Gamma$. Note that $G_{i+1}$ is also countable.  Now we define $G$ to be the union of $G_i$ for $i \geq 1$. It is easy to see that $G$ is countable, has no non-trivial normal amenable subgroups and that $G$ is non-C*-simple as we discussed in the beginning of the proof.
\end{proof}

Le Boudec result shows that the following corollary is not vacuous.

\begin{cor}
Let $G$ be a countable non-C*-simple group without non-trivial amenable subgroups.
There is a symmetric probability measure $\nu$ of full support on $G$ such that the boundary action $G \curvearrowright \partial(G, \nu)$ is not essentially free. Moreover, stabilizers of almost every point of the boundary are not normal subgroups and, in particular, they are not constant along the trajectories. 
\end{cor}
\begin{proof}
We take the measure $\nu$ constructed in theorem \ref{thm: main}.
Non-freeness is already proved. It remains to note that the action of $G$ on its Poisson boundary is amenable (since $\nu$ has full support) and hence the stabilizers of almost every point are amenable subgroups. There are no non-trivial amenable subgroups in $G$, so these stabilizers are not normal subgroups and they could not be constant along the trajectories.
\end{proof}

We note that most or even all previously known examples of non-free action of groups on their Poisson boundary(for non-degenerate measure) were associated with non-trivial normal amenable subgroup. All possibilities for stabilizers being normal subgroups are described by Erschler and Kaimanovich \cite{ErKa19}, following the idea of \cite{FHTF19}, namely, a normal subgroup could be a stabilizer of the action on the Poisson boundary iff it is amenable, and factor by this group is an ICC (infinite conjugacy class) group, provided that the measure on the group is non-degenerate (its support generates the whole group as semigroup).

The only previous example where stabilizers are not essentially constant (and hence normal subgroups) was given by A. Erschler and V. Kaimanovich (work in progress \cite{ErKa24+}, the result was previously announced on the conferences). They show that the infinite symmetric group has a measure such that the action on the Poisson boundary is totally non-free (a term coined by A. Vershik for actions where the map sending pints to their stabilizer subgroups is essentially one-to-one).

We also point out the following simple combinatorial corollary of the main theorem:

\begin{cor}
Let $G$ be a group. There is an amenable subgroup $H$ in $G$ such that for every finite amenably-visible subset $S$ we have $H^{\gamma}$ has non-empty intersection with $S$ for all $\gamma \in G$.
\end{cor}
\begin{proof}
There are at most countably many finite amenably-visible subsets in $G$, so we may take $H$ to be a stabilizer of almost any point of the Poisson boundary.
\end{proof}

It is worth noting a tangentially related work of Hartman and Kalantar where they characterized C*-simplicity via stationary traces on the reduced group C*-algebra \cite{HK23}. In that work they presented for C*-simple groups an example using and annalog of the Kaimanovich-Vershik construction, of a measure such that the unique stationary trace is the trivial group trace.

{\em Acknowledgements.} I would like to thank Vadim Koimanovich and Romain Tessera for discussions. I'm grateful to Anna Erschler for her comments and suggestions.   

\section{Notation}

For subsets $A, B$ of a group we denote $A^B = \lbrace a ^ b = b^{-1} a b \vert a \in A, b \in B\rbrace$. We say that a finite non-empty set $A$ is $(B, \varepsilon)$-invariant if $\lvert BA \setminus A \rvert < \varepsilon \lvert A\rvert$.
We will introduce necessary preliminaries on Poisson boundaries in section \ref{sec: stab nontriv}.

\section{Construction}\label{sec: construction}
Let $(\alpha_i)_{i \in \N}$ be a probability vector such that for an integer-valued i.i.d. process $(K_i)_{i \in \N}$ with individual distribution given by said vector, $\limsup_{i \to \N} W_i - i  = +\infty$ with probability $1$.
Let $(c_i)_{i \in \N}$ be any enumeration of the elements of $G$. 
Let us fix any enumeration of amenably-visible subsets $R_i$ such that for any amenably-visible subset $S$ and any $M > 0$ with for almost every realization of the i.i.d. process $(K_i)$ there are infinitely many $i$ such that $R_{K_i} = S$ and $K_i - i > C$. One way to construct such enumeration is to put any measure of full support on the collection of all amenably-visible subsets and take $R_i$ to be almost any realization of the resulting i.i.d. process. It will satisfy the requirement due to Fubini theorem. Let also $H_i$ be a sequence of amenable subgroups of $G$ corresponding to amenably-visible subsets $R_i$.

We iteratively define sets $A_n$ and $F_n$. First, let $A_1 = \lbrace 1_G \rbrace$.
\begin{enumerate}
\item $F_i$ is a symmetric $(R_i^{A_i^i} \cap H_i, 1/i)$-invariant subset of $H_i$;
\item $A_{i+1} = A_i \cup F_i \cup \lbrace c_i, c_i^{-1}\rbrace$.
\end{enumerate}

We define $\nu$ to be 
$$\sum_{i \in \N}  \frac{\alpha_i}{3}\big( \delta_{c_i} + \delta_{c^{-1}_i} + \lambda_{F_i}\big)$$.

\begin{lem}
For every $N > 0$, $S$ an amenably-visible set and $\varepsilon > 0$ there is $M$ such that if $m > M$ then $\nu^{*m}$ could be decomposed as
\begin{equation}\label{eq: decomposition}
\nu^{*m} = \eta + \alpha_{q',n,q''}\sum_{q',n,q''} q' \cdot \lambda_{F_n} \cdot q'',
\end{equation}
such that $\alpha_{q',n,q''} >= 0$, $\lvert \eta\rvert < \varepsilon$, $q' \in A_n^{n-1}$, $R_n = S$, $n > N$, and $q'' \in G$.
\end{lem}
\begin{proof}
This follows rather easily from the construction. Indeed, consider the process $(K_i)$ and an independent i.i.d. process $(Y_i)$ that takes value $\lbrace \text{``blue''}, \text{``red''}, \text{``green''}\rbrace$ with probabilities $(1/3, 1/3, 1/3)$. It is easy to observe that with probability $1$ there is $l > N$ such that $K_l > l+1$, $K_l > K_j$ for $j < l$, $R_{R_l} = S$ and $Y_l = ``blue''$. This implies that for big enough $M$ such $l$ with $l < M$ exists with probability bigger than $1 - \varepsilon$. Now we can couple the process $(K_i, Y_i)$ with the i.i.d. process $(X_i)$ where each $X_i$ has distribution $\nu$ individually for each $i$ if the following way. If $Y_i = \text{``blue''} $ we take $X_i$ unifromly from $F_{K_i}$. If $Y_i = \text{``red''}$, we set $X_i = c_{K_i}$, and if $Y_i = \text{``red''}$, we set $X_i = c^{-1}_{K_i}$. Now the decomposition follows naturally from the existence of this coupling.
\end{proof}

\begin{lem}\label{lem: asymptotic non-disjointness}
For any finite positive measure $\mu$ on $G$ and any amenably-visible subset $S$ of $G$ we have:
$$\liminf_{n \to \infty} \lvert t * \mu * \nu^{*n} - \mu * \nu^{*n} \rvert \leq 2 \Big(1 - \frac{1}{\lvert S\rvert}\Big)\lvert \mu \rvert,$$

for some $t \in S$. 
\end{lem}

\begin{proof}
We may assume that $\mu$ is finitely-supported. Let $W$ be its support. Take any $\varepsilon > 0$. Take $N$ such that $1/N < \varepsilon$ and $W$ is a subset of $A_N$.
Apply the previous lemma. We get an $M$ such for every $m > M$ a decomposition 
\[
\nu^{*m} = \eta + \alpha_{q',n,q''}\sum_{q',n,q''} q' \cdot \lambda_{F_n} \cdot q'',
\]
such that $\alpha_{q',n,q''} >= 0$, $\lvert \eta\rvert < \varepsilon$, $q' \in A_n^{n-1}$, $n > N$, and $q'' \in G$.

Now observe that $t \cdot wq' \cdot \lambda_{F_n} \cdot q'' = wq' \cdot t^{wq'} \cdot \lambda_{F_n} \cdot q''$,  for $w \in W$. Note that $wq' \in A^n_n$ (since $w \in A_n$ and $q' \in A_n^{n-1}$), so we get 
$$\lvert t \cdot wq' \cdot \lambda_{F_n} \cdot q'' - wq' \cdot \lambda_{F_n} \cdot q'' \rvert < 2\varepsilon,$$
as soon as $t^{wq'} \in H_n$. There exists such $t \in R_n = S$ by definition of the amenably-visible subset $R_n = S$ and since $H_i$ is an amenable subgroup that witnesses the amenable visibility. Now, by the pigeonhole principle, there is such $t \in R_n = S$ that 
$$\lvert t * \mu * \nu^{*n} - \mu * \nu^{*n} \rvert \leq 2 \Big(1 - \frac{1}{\lvert S\rvert}\Big)\lvert \mu \rvert + 3\varepsilon.$$
We now obtain the desired since $S$ is finite and the choice of $\varepsilon > 0$ was arbitrary.
\end{proof}

\section{Stabilizers are non-trivial}\label{sec: stab nontriv}
%We would like to to show that 
Let $G$ be a countable group and $\nu$ a measure on $G$. We say that measure is non-degenerate if its support generate $G$ as a semigroup.
Let $(X_i)_{i \in \N}$ be a $G$-valued i.i.d. process where each $X_i$ has distribution $\nu$. Let $(Y_i)_{i \in \N}$ be a process defined by $Y_i = Y_1 \cdot \ldots \cdot Y_i$. We endow the product $\Omega = \prod_{i = 1}^{\infty} G_i$, where $G_i$'s are copies of $G$, with the measure $\eta$ of path distribution of the $\nu$-random walk. In particular, $\pr_{G_i} \eta = \nu^{*i}$. There is an action of $G$ on $\Omega$: 
$$ \gamma \cdot (g_1, g_2, \ldots ) = (\gamma g_1, \gamma g_2, \ldots).$$
Let $\sA_i$ for $i \in \N$ denote the subalgebra of measurable subsets generated by the $G_i$ component of $\Omega$ (or equivalently, by $Y_i$). Let $\sA_{[i,\infty)}$ denote the join (the minimal subalgebra generated by) of $\sA_i, \sA_{i+1}, \ldots$.
The intersection $\sA = \bigcap_{i \in \N} \sA_{[i,\infty)}$ is the tail subalgebra of the random walk. Note that there is a unique up to isomorphism space $\partial(G, \nu)$ together with a natural map $\pr_{\partial} : \Omega \to \partial(G, \nu)$ such that $\sA_{\infty}$ is essentially the preimage of the natural Borel algebra on $\partial(G, \nu)$ under $\pr_{\partial}$. We note, that there is an induced quasi-invariant action of $G$ on $\partial(G, \nu)$ such that the map $\pr_{\partial} : \Omega \to \partial(G, \nu)$ is equivariant. 
In general, the Poisson boundary is defined as the space of ergodic components under the shift-action, but in case of non-degenerate measure, it coincides with the tail boundary we defined earlier (see \cite{Ka92}).

The next lemma will show us that measure $\nu$ constructed in the previous section is such that the stabilizer of almost every point of the Poisson boundary has non-empty intersection with every amenably-visible set 

\begin{lem}\label{lem: asymptotic disjointness}
Let $\nu$ be a measure of full support on $G$ such that 
If for a finite subset $S \subset$ of $G$ there is a positive-measure subset of points of the Poisson boundary $\partial(G,\nu)$ whose stabilizers are disjoint with $S$. Then for every $\varepsilon > 0$ there is positive finite measure $\mu$ on $G$ such that 
\begin{equation*}
\liminf_{n \to \infty}\lvert t * \mu * \nu^{*n} - \mu * \nu^{n} \rvert \geq (2 - \varepsilon) \lvert  \mu\rvert
\end{equation*}
\end{lem} 
\begin{proof}
We first note that there is a positive-measure subset $Q'$ of $\partial(G, \nu)$ such that $t \cdot Q'$ does not intersect $Q'$ for all $t \in S$. This corresponds to s subset $Q$ of $\Omega$. Let $\eta'$ be the restriction of $\eta$ to $Q$.
Now, using the same type of argument as in the proof of 0-2 law in \cite{Ka92}, we get that a projection $\mu_i = \pr_{G_i} \eta$ gives the desired for big enough $i$. This follows from the martingale convergence theorem similarly to the proof of the 0-2 law in \cite{Ka92}.

%Let $Q = %\pr^{-1}_{\partial}(Q')$. 
%Take any $t \in S$.
%We note that
%$$\lvert t \cdot \eta_Q - \eta_Q \rvert = 2 \lvert \eta_{Q}\rvert.$$  
%Let $\varphi$ be the Radon-Nikodym derivative of $d\eta_{Q}/d\eta$ and $\psi$ be the derivative $d(t\eta_Q)/d\eta$. Now $\varphi$ is in fact the indicator function of $Q$ while $\psi$ is not necessarily an indicator function. 
%We note that $\expect(\expect(\varphi \vert \sA_i) \vert \sA_{\infty})$ tends to $\varphi$ by the martingale convergence. We first note
\end{proof}

\begin{proof}[Proof of theorem \ref{thm: main}]
We take measure $\nu$ on group $G$ from section \ref{sec: construction}. Combination of lemmata \ref{lem: asymptotic disjointness} and \ref{lem: asymptotic non-disjointness} implies that for every finite amenably-visible subset $S$ of the group, stabilizers of the $G$ action on the Poisson boundary $\partial(G, \nu)$ have non-empty intersection with $S$.
\end{proof}

\end{document}